\documentclass[11pt,a4paper]{amsart}
\usepackage{amsmath, amscd, amssymb}

\usepackage[utf8]{inputenc} 
\usepackage[T1]{fontenc}

\usepackage{graphpap, color}
\usepackage[mathscr]{eucal}
\usepackage{mathrsfs}
\usepackage{pstricks}
\usepackage{cancel}
\usepackage[mathscr]{eucal}
\usepackage{verbatim}
\usepackage[all]{xy}
\usepackage{stmaryrd}
\usepackage[bookmarks, breaklinks, 
]{hyperref}

\def\cA{{\cal A}}

\numberwithin{equation}{section}

\newcommand{\CC}{\mathbb{C}}

\newcommand{\RR}{\mathbb{R}}

\newcommand{\cal}{\mathcal}

\def\cC{{\cal C}}
\def\cD{{\cal D}}
\def\cE{{\cal E}}
\def\cF{{\cal F}}

\def\cK{{\cal K}}

\def\cO{{\cal O}}
\def\cP{{\cal P}}

\def\begeq{\begin{equation}}
\def\endeq{\end{equation}}
\def\and{\quad{\rm and}\quad}

\def\and{\quad\text{and}\quad}

\DeclareMathOperator{\im}{im}

\newtheorem{prop}{Proposition}[section]

\newtheorem{theo}[prop]{Theorem}
\newtheorem{lemm}[prop]{Lemma}
\newtheorem{coro}[prop]{Corollary}

\newtheorem{defi}[prop]{Definition}
\newtheorem{conj}[prop]{Conjecture}

\newtheorem{defi-prop}[prop]{Definition-Proposition}

\def\dbar{\overline{\partial}}

\def\beq{\begin{equation}}
\def\eeq{\end{equation}}

\def\bee{\begin{equation}}
\def\eeq{\end{equation}}

\title{The limits of Kähler manifolds under holomorphic deformations}


\author[Mu-Lin Li]{Mu-Lin Li}
\address{School of Mathematics, Hunan University, China} \email{mulin@hnu.edu.cn}

\author[Wanmin Liu]{Wanmin Liu}
\address{
} \email{wanminliu@gmail.com}

\date{\today}

\begin{document}

\begin{abstract}
Under mild assumptions on the metric and topology of the central fiber, we prove that the limit of K\"ahler manifolds under holomorphic deformation is still K\"ahler.
\end{abstract}

\maketitle

\section{Introduction}
A {\it smooth family of compact complex manifolds} is defined as a proper holomorphic submersion  $f:X\to B$ between complex manifolds $X$ and $B$.  The following conjecture  regarding families of compact complex manifolds has been a longstanding problem.

\begin{conj} \label{conj-1}
Let $f:X\to\Delta=\{z\in\CC, |z|<1\}$ be a smooth family of compact complex manifolds such that all fibers $X_t:=f^{-1}(t)$ are Kähler for $t\in\Delta^\star:=\Delta\setminus \{0\}$. Then the central fiber $X_0:=f^{-1}(0)$ is a manifold in Fujiki class $\cC$.
\end{conj}

The case where $\dim X_t=2$ has been resolved  affirmatively, relying on the
theory of classification of surfaces and Siu's result \cite{Siu} (see also \cite{Bu}, \cite{La}).

 We list some results for the case $\dim X_t \ge 3$.  Hironaka \cite{Hi62} shows that the limit of Kähler manifolds of dimension three may not be Kähler again. Popovici \cite{pop1, pop2} proved that if $X_t$ is projective for $t\neq0$, then $X_0$ is Moishezon under some conditions on the Hodge numbers $h_{\dbar}^{0,1}(X_t)$ or on the metric on $X_0$. Under the same conditions,  Barlet \cite{Ba1} showed that if $X_t$ is Moishezon for $t\neq0$, then $X_0$ is also Moishezon. Rao and Tsai \cite{RT} proved that if there exist uncountably many Moishezon fibers in the family, then any fiber satisfying either of the above two conditions is still Moishezon.

When $X_t$ are Kähler manifolds for $t\neq0$ with $\dim X_t\geq3$, very few results are known, to the best of the authors' knowledge.  To establish that $X_0$ is Kähler, we focus on manifolds that satisfy specific cohomological and metric conditions, inspired by the aforementioned results. Our findings are as follows.

\begin{theo}\label{main-2}
Let $f:X \to \Delta$ be a smooth family of compact complex manifolds, where $X_t$ is Kähler for $t\neq0$. Assume the following conditions hold:
 \begin{itemize}
 \item $X_0$ admits a metric $\omega$ satisfying $\partial\dbar\omega=0$ and $\partial\omega\wedge\dbar\omega=0$.
 \item $h^{0,2}_{\dbar}(X_0)=h^{0,2}_{\dbar}(X_t)$ for $t\neq0$.
 \end{itemize}
 Then $X_0$ is a Kähler manifold.
\end{theo}

 The construction of non-K\"ahler compact complex manifolds equipped with metrics $\omega$ satisfying $\partial\dbar\omega=0$ and $\partial\omega\wedge\dbar\omega=0$  has attracted considerable attention in complex geometry. See the examples of such nilmanifolds in \cite[Section 2.3]{FT}.

As a corollary, we provide an alternative proof of the result that the limit of Kähler surfaces under holomorphic deformation is also Kähler (Corollary \ref{main-3}). Our proof does not rely on the classification theory of compact complex surfaces.

\begin{coro}\label{main-3}
Let $f:X \to \Delta$ be a smooth family of complex compact surfaces, where $X_t$ is Kähler for $t\neq0$. Then $X_0$ is also Kähler.
\end{coro}

The structure of this paper is as follows. In Section 2, we recall a criterion for determining whether a compact complex manifold $X$ is Kähler, provided it is equipped with a special metric $\omega$. In Section 3, we construct a special $d$-closed positive $(1,1)$-current $T$ on the central fiber $X_0$, ensuring that the volume $\int_{X_0}T_{ac}^n>0$, under certain numerical conditions on the dimensions of the cohomologies. This is achieved by applying the result of Demailly and  Păun \cite[Theorem 0.9]{Dem2}. In Section 4, we prove the main results of this paper by combining the findings from Sections 2 and 3.

\subsection*{Acknowledgement}

The authors express their gratitude to Xiao-Lei Liu, Sheng Rao, Mengjiao Wang and Zhiwei Wang
for their valuable discussions. The authors sincerely thank the referees for their invaluable comments on the paper.
M.-L. Li acknowledges support from the National Natural Science Foundation of China (NSFC), under grants No. 12271073 and No. 12271412.

\section{Currents on Complex Manifolds}


A compact complex manifold $X$ is said to be {\itshape in Fujiki class $\cC$} if there exists a complex Kähler manifold $Y$ and a surjective meromorphic map $h: Y\to X$. These manifolds were first introduced by Fujiki \cite{F1}.

Let $\cA^{i,j}$ denote the sheaf of smooth $(i,j)$-forms on $X$, and let $\cE^{i,j}(X):=\Gamma(X,\cA^{i,j})$ represent the global smooth $(i,j)$-forms on $X$.

Let $\omega$ be a smooth $(1,1)$-form on $X$, which can be locally written as $\omega=\sqrt{-1}h_{i\bar{j}}dz_i\wedge d\bar{z}_{j}$. The form $\omega$ is called {\itshape positive} (resp. {\itshape semi-positive}) if
$(h_{i\bar{j}})$ is a positive (resp. {\itshape semi-positive}) definite Hermitian matrix. The space of currents of degree $(p,q)$ over $X$ is denoted by $\cD^{p,q}(X)$. A current of degree $(p,q)$ can be regarded as an element $T$ in the dual space $(\cE^{n-p,n-q}(X))^\prime$.
A current $T$ of degree $(p,p)$ is said to be {\itshape positive} if,  for any choice of smooth $(1,0)$-forms $\alpha_1,\cdots,\alpha_{n-p}$ on $X$,
$$
 T\wedge \sqrt{-1}\alpha_1\wedge\overline{\alpha_1}\cdots \wedge \sqrt{-1}\alpha_{n-p}\wedge\overline{\alpha_{n-p}}
$$
is a positive measure. For a closed positive $(1, 1)$-current $T$ on $X$, let $T_{ac}$ denote the absolutely continuous part of $T$ in its Lebesgue decomposition $T = T_{ac} + T_{sg}$. So $T_{ac}$ is considered as a $(1, 1)$-form
with $L_{loc}^1$ coefficients, so that the product $T^m_{ac}$ is taken pointwise for $0\le m\le n$.

A {\itshape Kähler current} on a compact complex space $X$ is defined to be a closed positive current $T$ of bidegree $(1,1)$  satisfying $T\ge\epsilon \omega$, where $\epsilon$ is a positive number and $\omega$ is a smooth positive Hermitian form on $X$.

There are several other methods to guarantee that a manifold belongs to Fujiki class $\cC$. One of them is the following theorem,  which combines \cite[Theorem 5]{Va} and \cite[Theorem 3.4]{Dem2}.

\begin{theo}\label{equiv1}
Let $X$ be a compact complex manifold of dimension $n$. The
following conditions are equivalent.
\begin{itemize}
\item[(1)] $X$ is in Fujiki class $\cC$.
\item[(2)] There exists a Kähler current $T$ on $X$.
\end{itemize}
\end{theo}

If the compact complex manifold $X$ is equipped with a special metric, we can derive a criterion for $X$ to be Kähler.
\begin{theo}\label{cri}\cite[Theorem 1.2]{Wang}
Let $X$ be an $n$-dimensional compact complex manifold with a metric $\omega$ satisfying $\partial\dbar\omega=0$ and $\partial\omega\wedge\dbar\omega=0$. If there is a closed positive current $T$ on $X$ such that { $\int_XT_{ac}^n>0$}, where $T_{ac}$ is the absolutely continuous part of the Lebesgue decomposition of $T$, then $X$ is Kähler.
\end{theo}

\section{Cohomologies}
There are several different types of cohomology theories for compact complex manifolds. We consider three of them: the Bott--Chern cohomology, the Dolbeault cohomology and the de~Rham cohomology.

\subsection{Cohomologies of Complex Manifolds}

Let $\cE^{i,j}(X)=\Gamma(X,\cA^{i,j})$ denote the global smooth $(i,j)$-forms on $X$, where $\cA^{i,j}$ is the sheaf of smooth $(i,j)$-forms on $X$. Then, we have a double complex $(\cE^{\bullet,\bullet}(X)$, $\partial,\dbar)$.  Let $\cE^k(X)=\oplus_{i+j=k}\cE^{i,j}(X)$, and define the total complex $(\cE^{\bullet}(X), d=\partial+\dbar)$. There are two natural filtrations of the complex $(\cE^{\bullet,\bullet}(X),\partial,\dbar)$, namely, for $p,q\in\mathbb{N}$ and $k\in\mathbb{N}$:
\begin{equation}\label{seq}\nonumber
{^{\prime}F}^p(\cE^{k}(X)):=\bigoplus_{r+s=k, r\ge p} \cE^{r,s}(X),\quad \text{and} \quad {^{\prime\prime}F}^q(\cE^{k}(X)):=\bigoplus_{r+s=k, s\ge q} \cE^{r,s}(X),
\end{equation}
which induce two spectral sequences:
$$ ({^\prime E}^{\bullet,\bullet}_r,\ {^\prime d}_r)\quad\quad \text{and }\quad\quad({^{\prime\prime} E}^{\bullet,\bullet}_r,\  {^{\prime\prime}d}_r), \quad r\in\mathbb{N}.$$

These spectral sequences are called the {\itshape Hodge spectral sequence} and the {\itshape Frölicher spectral sequence}, respectively.
Both spectral sequences converge to the de~Rham cohomology $H^{\bullet}_{dR}(X,\CC)$.

For the differential complex $(\cE^{p,\bullet},\dbar)$, the {\em Dolbeault cohomology} is defined as
$$
H^{\bullet,\bullet}_{\dbar}(X,\CC):=\frac{\ker \dbar}{\im \dbar}.
$$
The {\em Bott--Chern cohomology} of $X$ is the bi-graded algebra
$$ H^{\bullet,\bullet}_{BC}(X,\CC) \;:=\; \frac{\ker\partial\cap\ker\dbar}{\im\partial\dbar}.$$
Let $h_{\dbar}^{p,q}(X)=\dim_{\CC} H^{p,q}_{\dbar}(X,\CC)$, $h^{k}(X)=\dim_{\CC} H^{k}_{dR}(X,\CC)$, and  $h_{BC}^{p,q}(X)=\dim_{\CC} H^{p,q}_{BC}(X,\CC)$.

The {\em Aeppli cohomology} of $X$ is the bi-graded $H^{\bullet,\bullet}_{BC}(X,\CC)$-module
$$ H^{\bullet,\bullet}_{A}(X,\CC) \;:=\; \frac{\ker\partial\dbar}{\im\partial+\im\dbar} \;.$$

The following theorem is well-known, and we refer to \cite[Lemma 2.5]{Sch} for a proof.

\begin{theo}\label{pair}
Let $X$ be a compact complex manifold of complex of dimension n. For any $(p,q)\in\mathbb{N}^2$, the wedge $\wedge$-operator induces a non-degeneration bilinear map
\beq
\begin{CD}
  H^{p,q}_{BC}(X,\CC)\times H^{n-p,n-q}_{A}(X,\CC)@>^\wedge>>H^{n,n}_{A}(X,\CC)@>^{\int_{X}}>>\mathbb{C}
\end{CD}\nonumber
\eeq
\beq
\begin{CD}
  (\alpha,\beta) @>>>\{\alpha\}\cdot\{\beta\}.
\end{CD}\nonumber
\eeq
\end{theo}

Fix a Hermitian metric on $X$, we can define the following differential operator,
$$
\Delta_{BC}:=(\partial\dbar)(\partial\dbar)^*+(\partial\dbar)^*(\partial\dbar)+(\dbar^*\partial)(\dbar^*\partial)^*+(\dbar^*\partial)^*(\dbar^*\partial)+\dbar^*\dbar+\partial^*\partial,
$$
which is a self-adjoint elliptic differential operator. Therefore, the following results hold.

\begin{prop}
Let $X$ be a compact complex manifold endowed
with a Hermitian metric. Then there exists an orthogonal decomposition

\beq\nonumber
\cE^{p,q}(X)=\ker\Delta_{BC}\oplus\im\partial\dbar\oplus(\im\partial^*+\im\dbar^*),
\eeq
and an isomorphism
\beq\nonumber
H^{\bullet,\bullet}_{BC}(X,\CC)\cong \ker\Delta_{BC}.
\eeq
\end{prop}
\begin{proof}
For a proof, we refer to \cite[Theorem 2.2]{Sch}.
\end{proof}

We also recall the classical result of Kodaira and Spencer concerning the upper semi-continuity of the kernel dimensions.
\begin{prop}
Let $f:X\to B$ be a smooth family of compact complex manifolds, and $X_b$ be the fiber of $f$ over $b\in B$. Then, for every $p\,,q\in\mathbb{N}$, the function
\beq\nonumber
b\in B\mapsto \dim_{\CC}H^{p,q}_{BC}(X_b,\CC)
\eeq
is upper semi-continuous.
\end{prop}
\begin{proof}
For a proof, we refer to \cite[Lemma 3.2]{Sch}.
\end{proof}

By applying Kodaira and Spencer's \cite[Theorem 5]{KS3} to the differential operator $\Delta_{BC}$, we obtain the following proposition.

\begin{prop}\label{sub-bu}
Let $f:X\to B$ be a complex analytic family of compact complex manifolds. If, for every $p\,,q\in\mathbb{N}$, the function $ \dim_{\CC}H^{p,q}_{BC}(X_b,\CC)$ is a constant function, then $\cup_{b\in B}H^{p,q}_{BC}(X_b,\CC)$ is a $C^{\infty}$-bundle over $B$.
\end{prop}

For a compact complex manifold $X$, we have the following short exact sequence
 \beq\label{sequence-1}
\begin{CD}H^{1,1}_{BC}(X,\CC)@>^\phi>>
H^2_{dR}(X,\CC)@>^u>>H^{2,0}_{\partial}(X,\CC)\oplus H^{0,2}_{\dbar}(X,\CC),
\end{CD}
\eeq
where \beq\nonumber
u:\alpha=\alpha^{2,0}+\alpha^{1,1}+\alpha^{0,2}\mapsto \alpha^{2,0}+\alpha^{0,2}
\eeq
is the projection. Then the following property holds.

\begin{prop}\cite[Lemma 3.3]{Sch}
Let $X$ be a compact complex manifold.  Then
 \beq
 h^2(X)\le 2h^{0,2}_{\dbar}(X)+h^{1,1}_{BC}(X).\nonumber
 \eeq
If $X$ is Kähler, the above inequality becomes an equality.
\end{prop}

\begin{theo}\label{family-1}
Let $f:X\to B$ be a smooth family of compact complex manifolds. Suppose that the fibers $X_b$ are Kähler for all $b\in B$ with $b\neq b_0$. If $h^2(X_{b_0})=2h^{0,2}_{\dbar}(X_{b_0})+h^{1,1}_{BC}(X_{b_0})$, then $\cup_{b\in  B}H^{1,1}_{BC}(X_b,\CC)$ is a $C^{\infty}$-subbundle of $R^2f_*\CC$ over $B$.
\end{theo}

\begin{proof} By the upper semi-continuity of Bott--Chern cohomology and Dolbeault cohomology, we have the following inequalities:
\begin{eqnarray*}
h^{0,2}_{\dbar}(X_b)&\le& h^{0,2}_{\dbar}(X_{b_0}),\\
  h^{1,1}_{BC}(X_b)&\le& h^{1,1}_{BC}(X_{b_0}).
\end{eqnarray*}

Thus for $b\neq b_0$,
\begin{eqnarray*}h^2(X_{b})&=&h^{0,2}_{\dbar}(X_b)+h^{2,0}_{\dbar}(X_b)+h^{1,1}_{BC}(X_b)\\
&=&2h^{0,2}_{\dbar}(X_b)+h^{1,1}_{BC}(X_b)\\
&\le& 2h^{0,2}_{\dbar}(X_{b_0})+h^{1,1}_{BC}(X_{b_0})\\
&=&h^2(X_{b_0})\\
&=&h^2(X_{b}).
\end{eqnarray*}
Thus,
\beq\nonumber
h^{1,1}_{BC}(X_b)= h^{1,1}_{BC}(X_{b_0}).
\eeq
It follows that $\cup_{b\in  B}H^{1,1}_{BC}(X_b,\CC)$ forms a $C^{\infty}$-bundle.

Since $h^2(X_b)=2h^{0,2}_{\dbar}(X_b)+h^{1,1}_{BC}(X_b)$ for all $b\in B$, the morphism $\phi$ in the exact sequence \eqref{sequence-1} is injective for every $X_b$. Therefore,  $\cup_{b\in B}H^{1,1}_{BC}(X_b,\CC)$ is a $C^{\infty}$-subbundle of $R^2f_*\CC$ over $B$.
\end{proof}

\subsection{Construction of Positive Current}
To construct a positive current on the central fiber, we first introduce some notation regarding the Barlet cycle space of compact complex manifolds.

\begin{defi} Let $f:X\to B$ be a smooth proper morphism. A smooth $B$-relative $(1,1)$-form $\omega_{X/B}$ is called a $B$-relative Kähler form if its restriction to each fiber $X_b$ is a Kähler form for all $b\in B$.
\end{defi}

Suppose that the fiber $X_b$ is Kähler. Then by \cite[Section 6]{KS3}, there exists an open neighborhood $U_b\subset B$ containing $b$, such that the restriction family $f^{-1}(U_b)\to U_b$ admits a $U_b$-relative Kähler form $\omega_{f^{-1}(U_b)/U_b}$.

Let $n=\dim X_b$ denote the dimension of the fiber $X_b$ for $b\in B$. For $0\le k\le n$, let $C^k(X/B)$ be the $B$-relative Barlet space, which parametrizes the $k$-dimensional cycles in $X$ over $B$. There exists a canonical morphism $$\mu:C(X/B):=\cup_kC^k(X/B)\to B.$$
For a base change $T\to B$, it follows that $$C(X\times_{B}T/T)=C(X/B)\times_{B}T.$$

We believe that the following properness property of the morphism $\mu$ is well-known to experts. However, since we could not find a suitable reference for this result, we state it here along with its proof.

\begin{prop}\label{prop-0}
Let $f:X\to B$ be  a smooth proper morphism with a $B$-relative Kähler form $\omega_{X/B}$. Suppose $A$ is a connected component of $C(X/B):=\cup_kC^k(X/B)$. Then the restricted morphism $\mu|_A:A\to B$ is proper.
\end{prop}
\begin{proof} The proof follows a similar approach to \cite[Proposition 3.0.6]{Ba1}, which addresses the case of relative cycle spaces $C^{n-1}(X/B)$.
Since $\omega_{X/B}$ is a  $B$-relative Kähler form on $X$, the form $\omega_{X/B}^k$ is closed when restricted to the fiber $X_b$, for $1\le k\le n$. For a $k$-dimensional relative cycle $Z$, define the quantity $Q(Z):=\int_Z\omega^k_{X/B}$. This is a continuous function on $Z$. On any given compact set $K\subset B$, the function $Q$ is bounded. Since $Q$ is locally constant on the fibers of the projection $C^k(X/B)\to B$, for a connected component $A$ of $C(X/B):=\cup_kC^k(X/B)$, the restricted morphism $\mu|_A:A\to B$ is proper by Bishop's theorem \cite[Theorem 1]{Bi}.
\end{proof}

\begin{prop}\label{prop-1}
Let $f:X\to B$ be a smooth proper morphism with fibers $X_b$ being Kähler for all $b\in B$. Suppose $A$ is a connected component of $C(X/B):=\cup_kC^k(X/B)$. Then the restricted morphism $\mu|_A:A\to B$ is proper.
\end{prop}
\begin{proof} Since the fibers $X_b$ are Kähler for all $b\in B$, there exists a small neighborhood $U\subset B$ of any $b\in B$ such that the restricted family $f_b:f^{-1}(U)\to U$ admits an $U$-relative Kähler form. Consequently, the induced morphism
\beq
\mu_b: C(f^{-1}(U)/U)\to U\nonumber
\eeq
is proper on when restricts on an arbitrary component. It is straightforward to verify that $C(f^{-1}(U)/U)=C(X/B)\times_BU$, and $\mu_b$ is the restriction of $\mu$ to $C(f^{-1}(U)/U)$. Therefore, for a connected component $A$ of $C(X/B)$, the map $A\times_BU\to U$ is proper. Thus the restricted morphism $\mu|_A:A\to B$ is proper.
\end{proof}

If all the fibers of the family $f:X\to B$ are Kähler manifolds, then the bundles
$$\cup_{b\in B}H^{1,1}_{BC}(X_b,\CC), \quad \cup_{b\in B}H^{0,2}_{\dbar}(X_b,\CC), \quad \cup_{b\in B}H^{2,0}_{\dbar}(X_b,\CC)$$
are all $C^{\infty}$-subbundles of $R^2f_*\CC$. Denote these subbundles by
$$
\cF_{B}^{p,q}, \text{ with } p+q=2, p\ge 0, q\ge 0.
$$
This gives the following $C^{\infty}$-decomposition:
\beq
R^2f_*\CC=\cF^{2,0}_B\oplus\cF^{1,1}_B\oplus\cF^{0,2}_B.\nonumber
\eeq
The Gauss--Manin connection $\nabla$ on $R^2f_*\CC$ can then be expressed as
\beq
\nabla=\left[ {\begin{array}{ccc}
    \nabla^{2,0} & *&* \\
    *& \nabla^{1,1}& * \\
    *&*&\nabla^{0,2}
  \end{array} } \right],\nonumber
\eeq
where $\nabla^{p,q}$ are the induced connections on the subbundles $\cF^{p,q}_B$.

Using the properness of connected component of the Barlet cycle space, Demailly and Păun \cite{Dem2} proved the following celebrated theorem on the Kähler cone.
\begin{theo}\cite[Theorem 0.9]{Dem2}\label{Kahler}
Let $f:X\to S$ be a deformation of compact Kähler manifolds over an irreducible base $S$. Then there exists a countable union $S'=\cup S_\nu$ of analytic subsets $S_\nu\subset S$, such that the Kähler cones $\cK_t\subset H^{1,1}_{BC}(X_t,\CC)$ are invariant over $S\setminus S'$ under parallel transport with respect to the $(1,1)$-projection $\nabla^{1,1}$ of the Gauss--Manin connection.
\end{theo}

This theorem can be used to produce a closed positive current on the central fiber.

\begin{lemm}\label{main0}
Let $f:X \to \Delta$ be a smooth family of compact complex manifolds, with $X_t$ being Kähler when $t\neq0$ and $\dim X_t=n$. Assume that $h^2(X_0)= 2h^{0,2}_{\dbar}(X_0)+h^{1,1}_{BC}(X_0)$. Then there exists a closed positive $(1,1)$-current $T$ on $X_0$ with $\int_{X_0}T_{ac}^n>0$.
\end{lemm}

\begin{proof}
 Let $\Delta^\star=\Delta\setminus \{0\}$, and consider the restriction $f':X'\triangleq f^{-1}(\Delta^\star)\to \Delta^\star$. The morphism $f'$ is smooth and all fibers are Kähler manifolds.

By Proposition~\ref{prop-1}, the morphism $\mu:C(X'/\Delta^\star):=\cup_kC^k(X'/\Delta^\star)\to \Delta^\star$ is proper for $0\le k\le n$. Denote by $\Sigma^{\prime\prime}$ the union of the image in $\Delta^\star$ of those components of $C(X'/\Delta^\star)$ that do not project onto $\Delta^\star$.  By \cite[Theorem]{F0} there are at most countably many components of $C(X'/\Delta^\star)$.  Since $\mu$ is proper, $\Sigma^{\prime\prime}$ is an union of countably many proper analytic closed subsets of $\Delta^\star$.

Let $\Sigma':=\Delta^\star\setminus \Sigma^{\prime\prime}$, and let $\Sigma:=\Sigma'\cup \{0\}=\Delta\setminus \Sigma^{\prime\prime}$. It follows that $\Sigma$ is arcwise connected by piecewise smooth analytic arcs. For $t_0\in \Sigma'$, consider a smooth arc
\beq\nonumber
\gamma: [0,1]\to \Sigma,\quad\quad u\mapsto \gamma(u),
\eeq
satisfying $\gamma(0)=0$ and $\gamma(1)=t_0$.

Since $f:X \to \Delta$ is a smooth family of compact complex manifolds, the sheaf $R^2f_*\CC$ is locally constant on $\Delta$. There exists a Gauss--Manin connection $\nabla$ on $R^2f_*\CC$. By the assumption $h^2(X_0)= 2h^{0,2}_{\dbar}(X_0)+h^{1,1}_{BC}(X_0)$, Theorem~\ref{family-1} implies that $\cF_{\Delta}^{1,1}=\cup_{t\in \Delta}H^{1,1}_{BC}(X_t,\CC)$ is a smooth subbundle of $R^2f_*\CC$.  Let $\cF_{\Delta}^{2}$ be the smooth orthogonal complement of $\cF_{\Delta}^{1,1}$ in $R^2f_*\CC$, so that $\cF_{\Delta,t}^{2}\cong H^{2,0}_{\dbar}(X_{t},\CC)\oplus H^{0,2}_{\dbar}(X_{t},\CC)$ for $t\in\Delta^\star$. Denote by $\nabla^{1,1}$ the restriction of the Gauss--Manin connection $\nabla$ to $\cF_{\Delta}^{1,1}$.

Let $\alpha(u)\in H^{1,1}_{BC}(X_{\gamma(u)},\CC)$ for $u\in [0,1]$ be a family of real $(1,1)$-cohomology classes obtained by parallel transport under $\nabla^{1,1}$ along $\gamma$. From the argument in the proof of \cite[Theorem 0.9]{Dem2}, we can assume that  $\alpha(u)$ is constant under the parallel transport for $u\in(0,1]$. This implies that $\nabla\alpha(u)\in H^{2,0}_{\dbar}(X_{\gamma(u)},\CC)\oplus H^{0,2}_{\dbar}(X_{\gamma(u)},\CC)$ for $u\in (0,1]$. Consequently,
\beq
\frac{d}{du}\bigg([X_{\gamma(u)}]\cdot\alpha^n(u)\bigg)=n[X_{\gamma(u)}]\cdot \alpha^{n-1}(u)\cdot\nabla\alpha(u)=0,\nonumber
\eeq
and thus
\beq
[X_{\gamma(u)}]\cdot\alpha^n(u)=c\nonumber
\eeq
is constant for $u\in(0,1]$.

Suppose that $\alpha(1)$ is a Kähler class on $X_{\gamma(1)}$ for $t_0=\gamma(1)\in \Sigma'$.
 By Theorem \ref{Kahler}, the Kähler cone $\cK_{\gamma(u)}\subset H^{1,1}_{BC}(X_{\gamma(u)},\CC)\cong H^{1,1}_{\dbar}(X_{\gamma(u)},\CC)$ of the fibers $X_{\gamma(u)}$ are invariant along $\gamma((0,1])$ under the parallel transport with respect to $\nabla^{1,1}$, thus the classes $\alpha(u)$ are Kähler classes for all $u\in (0,1]$.

Let $\omega_{\gamma(u)}$ be a smooth positive $(1,1)$-closed form in the class $\alpha(u)$ for $u\in (0,1]$. Let $g_t$ be a smooth family of Gauduchon metrics after possibly shrinking on $\pi:X\to \Delta$. Then
\beq\nonumber
\int_{X_{\gamma(u)}}\omega_{\gamma(u)}\wedge g_{\gamma(u)}^{n-1}=\alpha(u)\cdot \{g_{\gamma(u)}^{n-1}\},
\eeq
where $\{g_{\gamma(u)}^{n-1}\}\in H^{n-1,n-1}_{A}(X_{\gamma(u)},\CC)$ is the corresponding class.

Since $\dim_{\CC} H_{BC}^{1,1}(X_{\gamma(u)},\CC)$ is independent of $u\in [0,1]$, there is a $d$-closed smooth $(1,1)$-form $\Omega$ on $X|_{\gamma([0,1])}$ such that $[\Omega|_{X_{\gamma(u)}}]=\alpha(u)$ in $H^{1,1}_{BC}(X_{\gamma(u)},\mathbb{C})$ for any $u\in [0,1]$ by the proof of \cite[Theorem 9]{KS3} after replacing the Dolbeault Laplacian by $\Delta_{BC,t}$.
  Now we have that
 \beq
 \alpha(u)\cdot \{g_{\gamma(u)}^{n-1}\}=\int_{X_{\gamma(u)}}\Omega\wedge g_{\gamma(u)}^{n-1},\nonumber
 \eeq
which is point-wisely smoothly dependent on $u\in [0,1]$. Therefore $\alpha(u)\cdot \{g_{\gamma(u)}^{n-1}\}$ is finite for $u\in [0,1]$. Thus there exits a constant $C$ such that
\beq\nonumber
0<\int_{X_{\gamma(u)}}\omega_{\gamma(u)}\wedge g_{\gamma(u)}^{n-1}<C,
\eeq
for $u\in (0,1]$ after shrinking the disc small enough. Thus there exists a weakly convergent sequence $\omega_{\gamma(u_k)}\to T$ as $u_k\to 0$. The limit current $T$ is a positive $d$-closed $(1,1)$-currents. Let $ T_{ac}$ be the absolutely continuous part of $T$. The semi-continuity property \cite[Proposition 2.1]{Bou} for the top power of the absolutely continuous part in $(1, 1)$-currents shows that,
\beq\nonumber
T^n_{ac}\ge \lim \sup \omega_{\gamma(u_k)}^n.
\eeq
Thus
\beq
\int_{X_0}T^n_{ac}\ge\int_{X_{\gamma(u_{k})}}\omega_{\gamma(u_{k})}^n=c>0.\nonumber
\eeq

\end{proof}

\subsection{Cohomology under Special Metrics}
In this section we study the behaviors of the cohomologies of compact complex manifolds equipped with special  metrics.

\begin{prop}\label{derh}
Let $X$ be a compact complex manifold of dimension $n$, equipped with a metric $\omega$ such that $\partial\dbar\omega=\partial\dbar(\omega^2)=0$. Then for  $0\le p\le n$ the identity map
$$\iota_{p}:H^{p,0}_{\dbar}(X,\CC)\to H^{p}_{dR}(X,\CC),
$$
is well-defined and injective.  Similarly, the conjugate map
$$
\psi_p: H^{p,0}_{\dbar}(X,\CC)\to H^{0,p}_{\dbar}(X,\CC)
$$
is also well-defined and injective.
\end{prop}

\begin{proof}Since $\partial\dbar\omega=\partial\dbar(\omega^2)=0$, it follows that $\partial\dbar\omega^{p}=0$ for all $1\le p\le n$.
Let $z\in X$ and choose local coordinates $(z_1,\cdots,z_n)$ around $z$ such that the metric $\omega$ has the form
$$\omega=\sqrt{-1}\sum_j\gamma_j(z)dz_j\wedge d\bar{z}_j$$
where $\gamma_j(z)>0$. Let $\alpha=\partial \phi^{p,0}$, where $\phi^{p,0}$ is a $(p,0)$-form, and assume that $\dbar\alpha=\dbar\partial \phi^{p,0}=0 $. Assume that $\partial\phi^{p,0}=\sum_{I}f_{I}dz_{I}$ is a summation of $\dbar$-closed $(p+1,0)$-forms, where $I=(i_1,\cdots, i_{p+1})$ is an ordered subset of $\{1,\cdots,n\}$ with $|I|=p+1$ and $i_1<\cdots<i_{p+1}$. Then
\begin{align*}
 & \quad \partial\phi^{p,0}\wedge\dbar\overline{\phi^{p,0}}\wedge\omega^{n-p-1} \\
 =& \quad C\bigg(\sum_{I} |f_{I}|^2
\gamma_{n-I}(z) \bigg)(\sqrt{-1})^ndz_1\wedge d\bar{z}_1\wedge\cdots\wedge dz_n\wedge d\bar{z}_n
\end{align*}
where $\gamma_{n-I}(z):=\prod_{k\notin I}\gamma_k(z)$ and $C=(-1)^{\frac{|I|(|I|-1)}{2}}(-\sqrt{-1})^{|I|}$ is a nonzero constant.
So
\begin{align*}
 &  \quad C\int_{X} \bigg(\sum_{I} |f_{I}|^2\gamma_{n-I}(z)\bigg)(\sqrt{-1})^ndz_1\wedge d\bar{z}_1\wedge\cdots\wedge dz_n\wedge d\bar{z}_n \\
 = & \quad \int_{X}\partial\phi^{p,0}\wedge\dbar\overline{\phi^{p,0}}\wedge\omega^{n-p-1}=0.
\end{align*}
Thus $\alpha=\partial\phi^{p,0}=0$.

Now let $\alpha^{p,0}\in H_{\dbar}^{p,0}(X,\CC)$, then $\dbar\partial\alpha^{p,0}=0$. So the argument above shows that $\partial\alpha^{p,0}=0$.  Therefore, every $\dbar$-closed $(p,0)$-form is also $d$-closed.
The identity map
\beq\nonumber
\iota_{p}:H^{p,0}_{\dbar}(X,\CC)\to H^{p}_{dR}(X,\CC),
\eeq
is thus well-defined.

To prove that $\iota_{p}$ is injective, suppose that $\iota_{p}(\beta^{p,0})=0$ for some $\dbar$-closed $(p,0)$-form $\beta^{p,0}$. Then $\beta^{p,0}=\partial\gamma^{p-1,0}$, where $\gamma^{p-1,0}$ is a $(p-1,0)$-form. Using the same argument as above, we conclude that $\beta^{p,0}=0$.
Thus the map $\iota_{p}$ is injective.

Because the $\dbar$-closed $p$-form is also $d$-closed, the conjugate map
\beq\nonumber
\psi_p: H^{p,0}_{\dbar}(X,\CC)\to H^{0,p}_{\dbar}(X,\CC)
\eeq
is well-defined. To prove that $\psi_p$ is injective, suppose that $\psi_p(\beta)=\bar{\beta}=0\in H^{0,p}_{\dbar}(X,\CC)$ for some $\beta\in H^{p,0}_{\dbar}(X,\CC)$. Then there exists a form $\alpha^{0,p-1}\in \cE^{0,p-1}(X)$ such that $\bar{\beta}=\dbar\alpha^{0,p-1}$.
So by the same calculation as above, we have
 \beq\nonumber
 \int_{X}\partial\overline{\alpha^{0,p-1}}\wedge\dbar\alpha^{0,p-1}\wedge\omega^{n-p}=0,
 \eeq
which implies that $\beta=0$. Therefore the conjugate map $\psi_p$ is also injective.
\end{proof}

Let $\cP$ be the sheaf of real-valued pluriharmonic functions over a compact complex manifold $X$. We have the following fine resolution,
\beq\nonumber
\begin{CD}0@>>>\cP@>>>
(\cA_X)_{\RR}@>^{\sqrt{-1}\partial\dbar}>>(\cA^{1,1}_X)_{\RR}\\
@>^d>>(\cA^{2,1}_X\oplus \cA^{1,2}_X)_{\RR}@>^d>>\cdots.
\end{CD}
\eeq
From \cite[Page 20]{Mch}, we know that the  $\dim_{\RR}H^{1}(X,\cP)$ is finite, and
\beq\nonumber\dim_{\RR}H^{1}(X,\cP)=\dim_{\CC} H_{BC}^{1,1}(X,\CC).\eeq

On the other hand, consider the following short exact sequence.

\beq\nonumber
\begin{CD}0@>>>\RR@>>>
\cO_X@>>>\cP@>>>0.
\end{CD}
\eeq
This gives rise to the following long exact sequence on sheaf cohomology
\beq\nonumber
\begin{CD}0@>>>H^0(X,\RR)@>>>H^0(X,\cO_X)@>>> H^0(X,\cP)\\@>>>H^1(X,\RR)
@>>>H^1(X,\cO_X)@>>> H^1(X,\cP)\\@>>>H^2(X,\RR)@>^{\pi^{0,2}}>>H^2(X,\cO_X).
\end{CD}
\eeq
The first three terms in the sequence form a short exact sequence
\beq\nonumber
\begin{CD}0@>>>\RR@>>>
\RR^{\oplus 2}@>>>\RR@>>>0.
\end{CD}
\eeq
So we have
\beq\label{seq-4}
\begin{CD}0@>>>H^1(X,\RR)
@>>>H^1(X,\cO_X)@>>> H^1(X,\cP)\\
@>>>H^2(X,\RR)@>^{\pi^{0,2}}>>H^2(X,\cO_X).
\end{CD}
\eeq
The morphism $\pi^{0,2}$ is the projection of $H^2(X,\RR)\subset H^2(X,\CC)$ to the $(0,2)$-part $H^{0,2}_{\dbar}(X,\CC)$ through the isomorphism $H^{0,2}_{\dbar}(X,\CC)\cong H^2(X,\cO_X)$.
Let $H^{1,1}(X,\RR)=\mathrm{ker}\, \pi^{0,2}$ be the subgroup of $H^2(X,\RR)$ consisting of $d$-closed real $(1,1)$-forms. Then one has the
exact sequence of real vector spaces,
\beq\nonumber
\begin{CD}0@>>>H^1(X,\RR)
@>>>H^1(X,\cO_X)@>>> H^1(X,\cP)\\
@>>>H^{1,1}(X,\RR)@>>>0.
\end{CD}
\eeq

Thus, we have the following theorem.

\begin{theo}\cite[Theorem 6.3]{Mch}
Let $X$ be a compact complex manifold, then
\beq\nonumber
h^{1,1}_{BC}(X)=2h^{0,1}_{\dbar}-h^1(X)+\dim_{\RR} H^{1,1}(X,\RR).
\eeq
\end{theo}

\begin{prop}\label{prop-good} Let $f:X\to \Delta$ be a smooth family of compact complex manifolds. Assume that $h_{\dbar}^{0,2}(X_t)$ is constant for all $t\in\Delta$, and there exists a Hermitian metric $\omega$ on the central fiber $X_0$ satisfying $\partial\dbar\omega=\partial\dbar\omega^2=0$. Then
\beq\nonumber
h^2(X_0)= 2h^{0,2}_{\dbar}(X_0)+h^{1,1}_{BC}(X_0).
\eeq
\end{prop}

\begin{proof}
Since the Hermitian metric $\omega$ on $X_0$ satisfies $\partial\dbar\omega=\partial\dbar\omega^2=0$, by Proposition~\ref{derh}, we have the following inequality
\beq\nonumber
h_{\dbar}^{0,2}(X_0)\ge h_{\dbar}^{2,0}(X_0).
\eeq

From the assumptions we know that for $t\neq0$,

\beq\nonumber
h_{\dbar}^{2,0}(X_t)=h_{\dbar}^{0,2}(X_t)=h_{\dbar}^{0,2}(X_0).
\eeq

Thus
 \beq\nonumber
 h_{\dbar}^{2,0}(X_t)\ge h_{\dbar}^{2,0}(X_0).
 \eeq
On the other hand, by the upper semi-continuity of the Dolbeault cohomology, we know that
 \beq\nonumber
h_{\dbar}^{2,0}(X_t)\le h_{\dbar}^{2,0}(X_0).
\eeq
So
\beq\nonumber
h_{\dbar}^{2,0}(X_t)= h_{\dbar}^{2,0}(X_0).
\eeq
It follows that
\beq\nonumber
h_{\dbar}^{0,2}(X_0)= h_{\dbar}^{2,0}(X_0).
\eeq
Combining it with Proposition~\ref{derh}, the morphism  $H^{2,0}_{\dbar}(X_0,\CC)\to H^2_{dR}(X_0,\CC)$ is injective and the morphism $\psi_2:H^{2,0}_{\dbar}(X_0,\CC)\to H^{0,2}_{\dbar}(X_0,\CC)$ is an isomorphism. Consequently, the projection
\beq\nonumber
H^2_{dR}(X_0,\CC)\to H^2(X_0,\cO_{X_0})\cong H^{0,2}_{\dbar}(X_0,\CC)
\eeq
is surjective. Therefore, the projection $\pi^{0,2}: H^2(X_0,\RR)\to H^2(X_0,\cO_{X_0})$ in \eqref{seq-4} is also surjective, yielding the exact sequence:

\beq\label{seq-5}
\begin{CD}0@>>>H^1(X_0,\RR)
@>>>H^1(X_0,\cO_X)@>>> H^1(X_0,\cP)\\
@>>>H^2(X_0,\RR)@>^{\pi^{0,2}}>>H^2(X_0,\cO_{X_0})@>>>0.
\end{CD}
\eeq
Therefore,
\begin{eqnarray*}
\dim_{\RR}H^2(X_0,\RR)&=& \dim_{\RR} H^2(X_0,\cO_{X_0})+\dim_{\RR} H^1(X_0,\cP)\\
&&-\dim_{\RR} H^1(X_0,\cO_{X_0})+\dim_{\RR} H^1(X_0,\RR)\\
&=& 2h^{0,2}_{\dbar}(X_0)+\dim_{\RR} H^{1}(X_0,\cP)+h^1(X_0)-2h^{0,1}_{\dbar}(X_0).
\end{eqnarray*}

By \cite[Proposition 4.12]{RT}, the invariance of $h^{0,2}_{\dbar}(X_t)$ implies that $h^{0,1}_{\dbar}(X_t)$ is also invariant.
Hence for $t\neq0$,
\beq\nonumber
h^1(X_0)=h^1(X_t)=2h^{0,1}_{\dbar}(X_t)=2h^{0,1}_{\dbar}(X_0).
\eeq
So from the exact sequence \eqref{seq-5}, we have
\begin{eqnarray*}
h^2(X_0)=\dim_{\RR}H^2(X_0,\RR)&=& 2h^{0,2}_{\dbar}(X_0)+\dim_{\RR} H^{1}(X_0,\cP)\\
&=& 2h^{0,2}_{\dbar}(X_0)+h^{1,1}_{BC}(X_0).
\end{eqnarray*}

\end{proof}

\section{Proof of the Theorems}

\begin{proof}[Proof of Theorem~\ref{main-2}] Since $X_0$ admits a metric $\omega$ such that $\partial\dbar\omega=0$ and $\partial\omega\wedge\dbar\omega=0$, it follows that $\partial\dbar \omega^p=0$ for all $0\le p\le n$. By Proposition~\ref{prop-good}, Lemma~\ref{main0} and Theorem~\ref{cri}, we conclude that  $X_0$ is a Kähler manifold.
\end{proof}

\begin{proof}[Proof of Corollary \ref{main-3}]

Since $X_0$ is a compact complex surface, the Frölicher spectral sequence degenerates at $E_1$. By \cite[IV, Theorem 2.9,\,Proposition 2.10]{BHPV}, we have
\beq\nonumber H_{dR}^2(X_0,\CC)=\bigoplus_{p+q=2} \,'F^pH^2(X_0,\CC)\cap \overline{'F}^qH^2(X_0,\CC).\eeq Thus
\beq\nonumber
h^2(X_0)=h^{2,0}_{\dbar}(X_0)+h^{1,1}_{\dbar}(X_0)+h^{0,2}_{\dbar}(X_0).
\eeq
By the upper semi-continuity of the Dolbeault cohomology, for $t\neq 0$, we have

\beq\nonumber
h^{0,2}_{\dbar}(X_0)= h^{0,2}_{\dbar}(X_t).
\eeq

Additionally, the Gauduchon metric on $X_0$ satisfies $\partial\dbar \omega=0$. Consequently, by Lemma~\ref{main0} and Theorem \ref{cri}(or \cite[Theorem 0.2]{Ch}), we conclude that $X_0$ is a Kähler surface.

\end{proof}

\end{document}